\documentclass{ccr}
\usepackage{amsfonts,amsmath,amssymb,amsthm,mathtools}
\usepackage{mathtools}
\usepackage{newtxmath}
\usepackage{amsmath,amsthm,verbatim,amssymb,amsfonts,amscd, graphicx}
\usepackage{graphicx}
\usepackage{mathrsfs}
\title{Unique continuation for water waves and dispersive multiplier equations}
\authorsnames{Adrian Kirkeby}

\authorsaffiliations{{Simula Research Laboratory, Oslo, Norway. \\ adrian@simula.no }}

\abstract{
We show that if a solution to the water wave equation, for an arbitrary short time interval, is flat on an open set and the horizontal fluid velocity at the surface is zero on the same open set, then the wave must vanish everywhere for all times. In addition, we use a result from non-harmonic Fourier analysis to show that $(1+1 d)$ linear dispersive PDE with Fourier multipliers also have this unique continuation property, subject to a natural asymptotic growth condition on the multiplier symbol.}

\usepackage{tabularx}  
\usepackage{booktabs}  
\usepackage[mode=text]{siunitx} 
\usepackage{graphicx} 
\usepackage{csquotes}\MakeOuterQuote{"}  
\usepackage[hidelinks]{hyperref} 
\usepackage{cleveref}
\usepackage{enumitem}

\usepackage{biblatex}

\makeatletter
\renewcommand{\@seccntformat}[1]{}
\makeatother

\newtheorem{theorem}{Theorem}[section]
\newtheorem{proposition}[theorem]{Proposition}

\newtheorem{definition}{Definition}[section]

\begin{document}
\maketitle

\section{}

\subsection{Introduction}
In this brief article we give a short, self-contained and simple proof of the unique continuation property for the water wave equations based on the Zakharov--Craig--Sulem formulation. The result says that if the water surface is flat and the horizontal velocity is zero on an open set and for an arbitrary short time, then it must be flat everywhere and for all time. The proof that water waves have this property seems to have appeared first in (\cite{kenig2020uniqueness}). This work therefore serves as an addendum to that result, but differs in that the result is obtained from the unique continuation property of the Dirichlet-to-Neumann (DN) operator. 

Moreover, we consider the Cauchy problem for linear, periodic and dispersive multiplier equations $u_t = -i\omega(D)u$, and show that their solutions has the unique continuation property under the assumption that $$\frac{|\omega(k)|}{|k|} \to \infty  \quad \text{as } \quad |k| \to \infty. $$

The proofs for the water waves system and the dispersive multiplier equations are quite different in nature. For water waves, it is the rigidity imposed on the waves by the existence of an harmonic potential governing the flow in the interior of the water that gives rise to the property. In the Zakharov--Craig--Sulem formulation of the PDE, this manifests itself by the introduction of the non-local DN operator. For the linear multiplier equations, the asymptotic condition on the dispersion relation(i.e., the multiplier) implies that the phase velocity $|c(k)| = |\omega(k)|/|k| $ is unbounded, and this is not limited to non-local operators (e.g., it is true for $\omega(D) = D^2_x$).  The condition implies that for a compactly supported initial condition, there must be Fourier modes traveling arbitrarily fast, in contrast to the behavior of hyperbolic (non-dispersive) waves. \\
\newline
There are many works on unique continuation for dispersive PDE. In (\cite{kenig2020unique}), the authors study unique continuation for a class of of non-linear, non-local dispersive PDEs and the effect of non-locality is central. In (\cite{ehrnstrom2006unique}), unique continuation for traveling waves with vorticity is considered, and in (\cite{zhang1992unique}) unique continuation for the KdV-equation is shown. For a recent overview of unique continuation results for non-linear dispersive PDE, we refer to the paper (\cite{linares2022unique}) and references therein. 

Unique continuation results are often useful in control theory and inverse problems, where they are used to established local observability of systems, or equivalently, local uniqueness from measurements (Cf. \cite{tenenbaum2009fast} or \cite{kirkeby2023feynman}). \\
\newline
In Section 1 we give a brief introduction to the water wave equations and the Craig--Sulem--Zakharov formulation and state well-posedness result. We then show the unique continuation for the DN operator and for the water wave equation. In Section 2 we introduce a multi-dimensional version of Beurling's theorem for non-harmonic Fourier series and apply it to the dispersive multiplier equation.

\subsection{1: The water wave equation}
We assume the water is an incompressible and irrotational fluid, and let $\eta(X,t)$ be the surface wave amplitude and $\Phi(X,z,t)$ the velocity potential of the fluid. We denote the domain occupied by the fluid by $\Omega_t$: 
$$ \Omega_t = \left\{ (X,z) \in \mathbf{R}^{d+1} : - H_0 + b(X) < z < \eta(X,t) \right\}. $$  
Here, $H_0$ is the average depth, $b(X)$ represents the variable bottom topography, the horizontal dimension is $d = 1,2$, and we assume $-H_0 + b(X) < 0$. We will use the notation 
$X = (x_1,x_2) $, $\nabla = (\partial_{x_1},\partial_{x_2})^T$, $\nabla_{X,z} = (\partial_{x_1},\partial_{x_2}, \partial_z)^T $ and $\Delta_{X,z} = \nabla_{X,z} \cdot \nabla_{X,z}$. Also, $\partial_\nu  = \nu \cdot \nabla_{X,z} $, where $\nu$ is the unit normal on bottom surface $-H_0 + b(X)$. The equations governing the water waves are (Cf. Ch. 5 in \cite{ablowitz2011nonlinear}). 
\begin{equation}
    \begin{cases}
        \partial_t \eta + \nabla \Phi \cdot \nabla \eta  - \partial_z \Phi = 0, \quad X \in \mathbf{R}^{d}, z = \eta,  \\
        \partial_t \Phi  + \frac{1}{2}|\nabla_{X,z} \Phi|^2 + g\eta = 0,  \quad X \in \mathbf{R}^{d}, z = \eta, \\
        \Delta_{X,z} \Phi  = 0, \quad (X,z) \in \Omega_t, \\
        \partial_{\nu}\Phi = 0, \quad \quad X \in \mathbf{R}^{d}, z = -H_0 + b(X). 
    \end{cases}
    \label{eq: 2d3d} 
\end{equation}

It is convenient to recast the free boundary problem above in terms of the velocity potential on the boundary. This transforms \eqref{eq: 2d3d} from a 2D+3D system to a 2D+2D system, at the cost of introducing the non-local DN operator, and is known as the Zakharov--Craig--Sulem formulation (Cf. \cite{waterwavesprob}). 

\begin{equation}
    \begin{cases}
    \partial_t \eta - G(\eta,b)\varphi = 0, \\
     \partial \varphi + \eta g  + \frac{1}{2}|\nabla \varphi |^2 - \frac{\left(G(\eta,b)\varphi + \nabla\varphi \cdot \nabla \eta\right)^2}{2(1 + |\nabla \eta|^2)} = 0, \\
    \left(\eta(0,X), \varphi(0,X)\right) = \left(\eta_0,\varphi_0\right).
    \end{cases}
    \label{full1}
\end{equation}
In \eqref{full1}, the DN operator $G(\eta,b)$ maps the surface potential $\varphi$ to the fluid velocity in the direction normal to the surface. It is defined as follows: 
Let $\Phi$ be the solution to the Dirichlet problem 
\begin{equation}
     \Delta \Phi = 0, \quad   \text{in } \Omega_t, \quad \Phi\big|_{z=\eta} = \varphi, \quad  \partial_\nu \Phi\big|_{-H_0+b(X)} = 0. 
\end{equation}
Then $G(\eta,b) \varphi = \sqrt{1 + |\nabla \eta|^2 } \partial_\nu \Phi\big|_{z=\eta} $. \\

Remarkably, and despite its formidable appearance, it has been showed that there exists a unique solution to the Cauchy problem for water waves under various assumptions. For an account of the different contributions see, e.g., (\cite{waterwavesprob} or \cite{alazard2014cauchy}). We rely here on the existence and uniqueness result from (\cite{alazard2014cauchy}), and give now a brief recapitulation of the result. 

\begin{itemize}
    \item The incompressible Euler equation, from which \eqref{eq: 2d3d} is derived is 
    \begin{equation*}
        \partial_t V + (V \cdot \nabla_{X,z})V  + \frac{1}{\rho_0} \nabla_{X,z}P = -g\mathbf{e}_z, \quad \nabla_{X,z} \cdot V = 0, 
    \end{equation*}
    where $V$ is the fluid velocity, $P$ is the pressure, $\rho_0$ is the mass density and $g$ is gravitational acceleration. 
    Define $a(X,t) = -\partial_z P(X,\eta(X,t),t)$. The Taylor sign condition is that there exists some constant $c > 0$ such that 
    $a(X,t) \geq c$, and is required for stability of the waves. 
    \item The water depth $\eta(X,t) - (-H_0 + b(X)) \geq h > 0$. 
    \item Define $$ B(\eta,\varphi)  = \frac{\nabla \eta \cdot \nabla \varphi - \mathcal{G}(\eta,b) \varphi}{1 +|\nabla \eta|^2}  \quad  \text{and } \quad V(\eta,\varphi) = \nabla \varphi - B(\eta,\varphi)\nabla \varphi.$$
\end{itemize}

\begin{theorem}[Thm. 1.2, \cite{alazard2014cauchy}]
    Assume $s > 1 + d/2$, and take initial data $(\eta_0,\varphi_0)$ such that 
    \begin{itemize}
    \itemindent=12pt
        \item $(\eta_0,\varphi_0) \in H^{s+1/2}(\mathbf{R}^d) \times H^{s+1/2}(\mathbf{R}^d), \quad  V_0, B_0 \in H^{s}(\mathbf{R}^d). $
        \item There is some $h > 0$ such that at $t = 0$ it holds that 
        $$\{ (X,z): X\in \mathbf{R}^d,  \eta(X,t) - h < z < \eta(X,t) \} \subset \Omega_t.$$ 
        This means the water depth is positive everywhere initially.
        \item There is some $c > 0$ such that $a(X,0) > c$, i.e., the Taylor condition holds initially.  
    \end{itemize}

    Then there is some $T > 0$ such that there exists a  unique solution $(\eta,\varphi)$ to equation \eqref{full1} that satisfy 

    $$ (\eta,\varphi) \in C([0,T]; H^{s+1/2}(\mathbf{R}^d) \times H^{s+1/2}(\mathbf{R}^d) ). $$
    \label{thm: alazard}
    \end{theorem}

\subsubsection{Unique continuation for the Dirichlet-to-Neumann operator}
Unique continuation results for non-local operators have recently been studied in connection with the fractional Laplacian $(-\Delta)^s$ and related equations, and are known to be hold in various settings (\cite{garcia2020two}, \cite{ruland2015unique}). In fact, for the case of infinite depth and flat surface, the DN operator is equal to the fractional Laplacian with $s = 1/2$: Recall that $(-\Delta)^s u = \mathcal{F}^{-1}\left(|\xi|^{2s} \mathcal{F} u\right)$, where $\mathcal{F}$ denotes the Fourier transform. By taking $\widehat{\Phi}(\xi,z)$ to be the Fourier transform of $\Phi$ in the horizontal direction, it is straight forward to show that $\widehat{\mathcal{G}\varphi} = \partial_z \widehat{\Phi}(\xi,z)|_{z=0} = |\xi|\widehat{\varphi}$.   In (\cite{ghosh2020calderon}, Thm. 1.2) it is shown that for any $s \in (0,1)$ and $u \in H^{r}(\mathbf{R}^d), r\in \mathbf{R}$, if $(-\Delta )^su = u = 0$ on an open set, then $u = 0$. For $s = 1/2$, this is the result we are looking for, but restricted to infinite depth and flat surface. 

We now show that if the boundary $\eta$ and bottom $b(X)$ is sufficiently regular, and if $\mathcal{O}$ is some open set where $\eta = 0$, then it holds that 
if $\Phi|_{\mathcal{O}} = \partial_\nu \Phi_{\mathcal{O}} = 0$, then $\Phi = 0 $ everywhere. This property is well known on bounded domains, and a proof is outlined in (\cite{kenig2020uniqueness}). We give a simple proof tailored for our use here. 

\begin{proposition}
 Let $k > d/2$ and $\eta, b \in H^{k + 1}(\mathbf{R}^d)$. Assume $\varphi \in \dot{H}^{3/2}(\mathbf{R}^d)$, and let $\mathcal{G}(\eta,b)$ be the corresponding Dirichlet-to-Neumann operator. Let $\mathcal{O} \subset \mathbf{R}^d$ be an open set, and assume $\eta|_\mathcal{O} = 0$. Then 
 $$ \varphi|_\mathcal{O} = \mathcal{G} (\eta,b)\varphi \big|_\mathcal{O} = 0 \quad \text{implies} \quad \varphi = 0.$$
 \label{thm: unique cont d2n}
\end{proposition}

\begin{proof}
    We first recall some results on harmonic functions. 
\hspace{15pt} 
\begin{itemize}[leftmargin=0.8cm]
    \item Let $\Omega$ be an open set in $\mathbf{R}^d$ and assume u satisfies $\int_\Omega u \Delta g \mathrm{d}x$ for all $g \in C^\infty_0(\Omega).$ Then $u$ has a harmonic representation in $\Omega$ (Ch. 2.18 in \cite{dagmar2018laplace}). 
    \item Let $\Omega$ be a bounded domain with Lipschitz boundary $\partial \Omega$ in $\mathbf{R}^d$, and let $\omega$ be an open subset of $\partial \Omega$. If $u \in H^2(\Omega)$ is harmonic and $ u = \partial_\nu u = 0$ on $\omega$, then $u = 0$ in $\Omega$ (Ch. 2.3 in \cite{choulli2016applications}). 
    \item Let $\Omega$ be an open set in $\mathbf{R}^d$ and assume u is harmonic in $\Omega$. Let $\mathcal{O}$ be an open subset of $\Omega$. If $u = 0$ in $\mathcal{O}$ then $u = 0$. 
\end{itemize}
Consider now the problem for the velocity potential. 
\begin{equation}
    \begin{cases}
     \Delta_{X,z} \Phi = 0, \quad   \text{in } \Omega_t, \\
     \Phi = \varphi, \quad \text{on } \eta, \\
     \partial_{\nu}\Phi  = 0, \quad \text{on } -H_0+b(X).  
    \end{cases}
    \label{eq: Phi}
\end{equation}
From proposition 2.44 in (\cite{waterwavesprob}), we know that for $k > d/2$ and $\eta,b \in H^{k+1}(\mathbf{R}^d)$ and $\varphi \in H^{3/2}(\mathbf{R}^d)$ there exist a unique solution $\Phi \in H^2(\Omega_t)$ to  \eqref{eq: Phi}, and by $i)$ above  $\Phi$ is harmonic in $\Omega_t$. 

Now, let $\mathcal{O} \subset \mathbf{R}^d$ be open, and assume $\eta|_\mathcal{O} = 0$. Then $\mathcal{G}(b,\eta)\varphi \big|_\mathcal{O} = \partial_z \Phi|_{X \in \mathcal{O},z = 0}$. Let $B_r$ be a ball of radius $r$ such that its intersection with $\eta $ is contained in $\mathcal{O}$ and that $B_t = B_r \cup \mathcal{O}_t$ is open. By assumption $\varphi|_{\mathcal{O}} = \mathcal{G}(b,\eta)\varphi \big|_\mathcal{O} = 0$.
Since $\Phi$ is harmonic in the domain $B_t$ and since $\partial B_t$ is Lipschitz, it follows that $\Phi = \partial_\nu \Phi = 0$ on $\partial B_t \cup \mathcal{O}$, and so $\Phi = 0$  in $B_t$. By the third property, this implies that $\Phi$ vanishes everywhere, and so $\varphi = \Phi\big|_{ z= \eta} = 0$.
    
\end{proof}

\subsubsection{Unique continuation for the water waves}

The unique continuation for water waves is now a simple consequence of applying Proposition \ref{thm: unique cont d2n} to \eqref{full1}. 

\begin{definition}
We say that the water waves system $(\eta,\varphi)$ is \emph{at rest} on an open set $\mathcal{O} \subset \mathbf{R}^d$ at time $t$ if $$(\eta(X,t), \nabla_X \varphi(X,t)) = 0 \quad \text{for } X \in \mathcal{O}.$$
\end{definition}

\begin{proposition}
Let $(\eta,\varphi)$ be the solution to the initial value problem \eqref{full1} satisfying the hypothesis of Theorem \ref{thm: alazard}. Let $\mathcal{O} \subset \mathbf{R}^d$ be an open set and assume $(\eta,\varphi)$ is at rest on $\mathcal{O}$ for some arbitrary time interval $I = (t_0,t_1) \subset [0,T]$. Then $(\eta,\varphi) = 0$. 
\end{proposition}

\begin{proof}
Take $Q_\epsilon = B_\epsilon \times I_\epsilon \subset \mathcal{O} \times I$ to be a space-time cylinder contained in $\mathcal{O} \times I$. 
Then $\partial_t \eta |_{Q_\epsilon} = 0 $, and so $\mathcal{G}(b,\eta) \varphi \big|_{Q_\epsilon} = 0$ by equation 1 in \eqref{full1}. 
Moreover, by the assumption that $\nabla_X \varphi = 0$ and equation 2 in \eqref{full1}, we see that $\varphi$ must be constant on $Q_\epsilon$. 
It follows now that $\Phi$ is constant: when $\Phi$ is harmonic and constant $(\Phi=C)$ on some open set, the difference $\Phi - C$ is zero on the open set and hence vanish everywhere and so $\Phi = C$. Therefore, by proposition \ref{thm: unique cont d2n}, $\Phi = C$ for some constant $C$. But this implies $\varphi = C$, and since $\varphi \in H^{s + 1/2}(\mathbf{R}^d)$, $\varphi = 0$. Consequently, $\eta$ is also constant, and by the same argument $\eta = 0$.   
\end{proof}

As a corollary, we also get the following result. 

\begin{proposition}
    Let $\mathcal{O} \subset \mathbf{R}^d$ be an open set and assume at some instant $t_1 \in [0,T]$  $(\eta,\varphi)$ is at rest on $\mathcal{O}$ and that in addition $\partial_t \eta(t_1,X)|_\mathcal{O} = 0$. Then $(\eta,\varphi) = 0$.
\end{proposition}

These results seem both natural and unnatural: on the one hand, it seems plausible that for a fully developed and chaotic wave field, like the ocean during a storm or a pool full of kids playing around, it would be very rare to find a small patch of the surface completely at rest. It is in a sense the intuition for incompressibility; if one part of a fluid moves, then the rest has to move too. On the other hand, a local perturbation of a surface at rest seemingly causes waves with finite speed of propagation (at least when considering gravitational effects only) and therefore should not have this kind of unique propagation property.  

\section{}
\subsection*{2: Unique continuation for linear dispersive equations}

We now consider linear dispersive PDE in one spatial dimension with periodic boundary conditions. In this setting we can derive unique continuation results directly from the dispersion relation. 
For a continuous function $\omega(k): \mathbb{R} \to \mathbb{R}$ we define the periodic Fourier multiplier $$\omega(D)u(x) = \sum_{k\in \mathbb{Z}} \omega(k) \hat{u}_k e^{ikx},$$ 
where $\hat{u}_k$ are the Fourier coefficients of $u$. For a given $\omega$, we consider the PDE 
\begin{equation}
\begin{cases}
    u_t = -i\omega(D)u, \quad x \in (0,2\pi), t > 0, \\
    u(x,0) = g(x), u(0,t) = u(2\pi,t). 
\end{cases}
\label{lin disp}
\end{equation}
We assume that there exists some $m > 0$ such that for $$|d^n_k\omega(k)| \leq C_n (1 + k^2)^{(m-n)/2}  \quad \text{for all } k \in \mathbb{R}.$$ 
Then, for initial data $g \in L^2(0,2\pi)$ the unique solution to \eqref{lin disp} is 
\begin{equation}
    u(x,t) = \sum_{k\in \mathbb{Z}}  \hat{g}_k e^{i(kx - \omega(k)t)}. 
\end{equation}
See, e.g., (\cite{craig2018course}) for a good introduction and well-posedness analysis of such equations. 
The equation \eqref{lin disp} is called dispersive if $\omega(k)$ is real valued and $\omega''(k) \neq 0$ (Cf. Ch. 2, \cite{ablowitz2011nonlinear}). For example, $\omega(k) = k^2$ corresponds to the free-space Schrödinger equation, $\omega(k) = -k^3$ to the linearized KdV equation and $\omega(k) = \sqrt{(gk + Sk^3)\tanh(kH)}$ to linear gravity-capillary waves in water of depth $H$. 
In the two first instances, $\omega$ is a local operator, while for the gravity-capillary waves it is non-local. 

The following proposition shows that solutions to \eqref{lin disp} have the unique continuation property under rather mild and natural assumptions on the dispersion relation. 

\begin{proposition} Assume that $\omega(k)$ is dispersive and satisfies 
$$ \frac{|\omega(k)|}{|k|} \to \infty \quad \text{as} \quad  |k| \to \infty.$$
Let $u$ be the solution to \eqref{lin disp}, and let $\mathcal{O} \in (0,2 \pi) \times \{ t \in \mathbb{R}: t >0\}$ be an open set. 
Then 
$$ u(x,t)|_{\mathcal{O}} = 0 \quad \text{implies} \quad  u = 0. $$
\label{ucp_lin}
\end{proposition}

To prove the above result, we will rely on a multi-dimensional version of Beurling's theorem for non-harmonic Fourier series (Cf. (\cite{young2001introduction}) for an introduction to non-harmonic Fourier series, and (\cite{baiocchi2002ingham}) for an introduction to Beurling's result.). We first introduce some terminology. 
\begin{itemize}
    \item Let $\Lambda = \{\lambda_m\}_{m\in I} \subset \mathbb{R}^n$ be a countable set of points. 
    \item The set $\Lambda$ is said to regular if $$ \underset{\substack{m,n \in I \\ m\neq n}}\inf|\lambda_m - \lambda_n| = \gamma > 0.$$
    \item Let $\{a_m\} \in \ell^2$ and define the function 
    $$f(x) = \sum_{m \in I}a_m e^{i \lambda_m \cdot x}.$$
    We then say that a domain $D \subset \mathbb{R}^n$ is a domain associated with $f$ if there exists constant $0 < d_- \leq d_+$ such that 

    $$ d_-\sum_{m \in I}|a_m|^2 \leq \int_D |f(x)|^2 \mathrm{d}x \leq d_+ \sum_{m \in I}|a_m|^2.$$
\end{itemize}

Clearly, if we can show that any bounded open set $\mathcal{O} \subset \mathbb{R}^n$ is a domain associated with such $f$, then we have shown the unique continuation property. To this end, we will utilize the following, multidimensional version of Beurling's theorem, characterizing such domains based on $\Lambda$. 

\begin{theorem}(Theorem 8.4.5, \cite{observation})
Let $\Lambda$ be a regular sequence and $B_r(x)$ be a ball with radius $r$ centered at $x$. Define 
$$ N(r) = \sup_{|x| \to \infty}  \# (\Lambda \cap B_r(x) ),$$ 
where $\# (\Lambda \cap B_r(x) )$ is the number of elements in the set $(\Lambda \cap B_r(x) )$.  
If $$N(r)/r \to 0 \quad  \text{as} \quad r \to \infty,$$ then any open set $D \subset \mathbb{R}^n$ is a domain associated with $\Lambda$. 
\label{tucsnak}
\end{theorem}

We can now prove Proposition \ref{ucp_lin}. 
\begin{proof}
Let $\Lambda = \{\lambda_k\}_{k\in \mathbb{Z}}$ with $\lambda_k = (k,\omega(k))$. Then 
$$ u(x,t) = \sum_{k\in \mathbb{Z}} \hat{g}_k e^{i \lambda_k \cdot(x,t)}.$$ 
Since
$$ \underset{\substack{m,n \in \mathbb{Z} \\ m\neq n}}\inf|(m,\omega(m)) - (n,\omega(n))| \geq 1,$$
the set $\Lambda $ is regular. 
We define $\Lambda_+ = \Lambda_{k \geq 0}$ and $\Lambda_- = \Lambda\setminus \Lambda_+$, and the functions 
$$\tilde{N}_\pm(|x|,r) = \sup_{ |y| = |x| } \frac{\#(\Lambda_\pm \cap B_r(y))}{r}.$$
Next, consider the annulus $\mathcal{A}(|x|,r) = \{ y \in \mathbb{R}^2:  |x| -r \leq  |y| \leq |x| + r \}$. 
For any fixed  $|x|$ and $r$, it is clear that 
$$ \tilde{N}_+(|x|,r) \leq \frac{\#(\Lambda_+ \cap \mathcal{A}(|x|,r))}{r}. $$
Assume now that $\omega(k) > 0$ for $k > 0$. Let $(k_-, \omega(k_-))$ and $(k_+, \omega(k_+)) $ be the coordinates of intersection with the inner and outer boundaries of $\mathcal{A}(|x|,r)$, respectively. We then have that $ \#(\Lambda_+ \cap \mathcal{A}(|x|,r)) \leq 1 + (k_+ - k_-)$. For $k$ large enough, $\omega(k) > k$, and the difference ${\omega(k_+) - \omega(k_-)}$ is less than the length $D$ of the longest vertical line segment  contained in ${\mathcal{A}(|x|,r) \cap \{ (y_1,y_2) \in \mathbb{R}^2 : , y_1 \geq 0, y_2 \geq y_1\}}$, i.e., in the intersection of the annulus and the points above the line $y = k$. Figure \ref{fig1} depicts the setting.  
The vertical distance between the inner and outer boundary of $\mathcal{A}$ is  
$$ d(k) = \sqrt{(|x| +r)^2 - k^2} - \sqrt{(|x| -r)^2 - k^2} $$ 
and since $d'(k) \geq 0$, we find that $D = \sqrt{(|x| +r)^2 - \frac{(|x| -r)^2}{2}} - \sqrt{\frac{(|x| -r)^2}{2}}$.
\begin{figure}
    \centering
    \includegraphics[width = 0.8\textwidth]{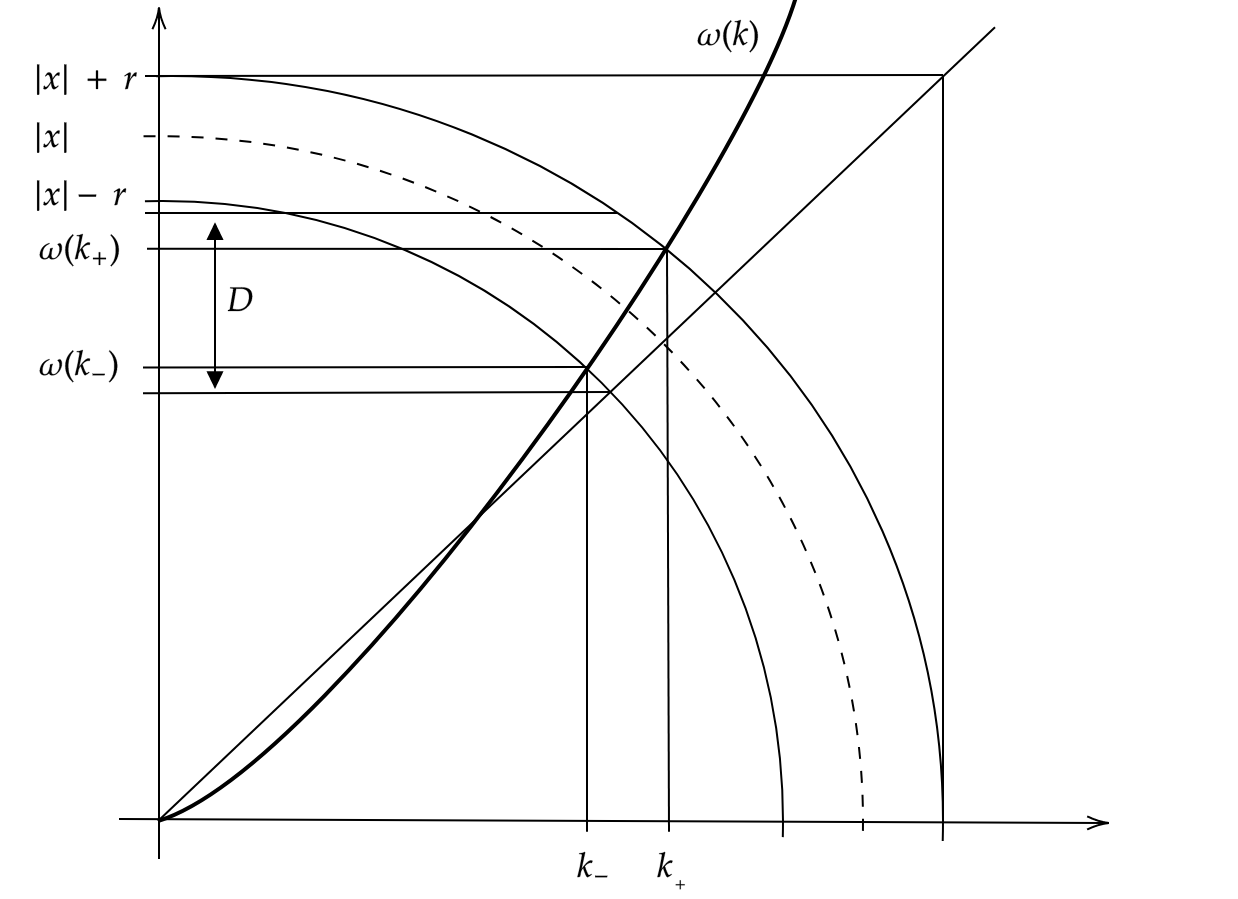}
    \caption{}
    \label{fig1}
\end{figure}
Further calculation then shows that $D$ is an increasing function of $|x|$ and that $\lim_{|x| \to \infty} D = \sqrt{8}r.$
By the mean value theorem, there is some ${\tilde{k} \in (k_-,k_+)}$ such that 
$$ k_+ - k_- = \frac{\omega(k_+) - \omega(k_-)}{\omega'(\tilde{k})} \leq \frac{\sqrt{8}r}{\omega'(\tilde{k})}. $$
For each $r$ we therefore have 
$$ \tilde{N}_+(|x|,r) \leq  1/r + \frac{\sqrt{8}}{\omega'(\tilde{k})}. $$
The value $k_-$ must be an increasing function of $|x|$, and since 
$\tilde{k} \geq k_-$, we have $\omega'(\tilde{k}) \to \infty$ as $|x| \to \infty$. Consequently $$\lim_{|x| \to \infty} \tilde{N}_+(|x|,r) = 1/r. $$
We now calculate $\lim_{r \to \infty} \tilde{N}_{+} (|x|,r).$ Let $\mathcal{R}(|x|,r)$ be a square in the first quadrant with its lower left corner at the origin, and side length $|x| + r$. For fixed $|x|$ large enough, we have that $$ \tilde{N}_+(|x|,r) \leq  \#(\Lambda_+ \cap \mathcal{R}(|x|,r))  \leq \omega^{-1}(|x| + r).$$ Since $\omega$ is superlinear, $\omega^{-1}$ is sublinear. Therefore  
$$\lim_{r \to \infty} \tilde{N}_+(|x|,r) \leq \lim_{r \to \infty}  \frac{\omega^{-1}(|x| + r)}{r} = 0.$$
We now have that 

\begin{equation}
\begin{split}
\lim_{|x| \to \infty } \lim_{r \to \infty} \tilde{N}_+(|x|,r) &= 0, \\
\lim_{r \to \infty } \lim_{|x| \to \infty} \tilde{N}_+(|x|,r) & = 0. 
\end{split}
\label{N+}
\end{equation}
By the same argument, \eqref{N+} also holds for $\tilde{N}_-$. We now have that 
$$ \lim_{r\to \infty} N(r)/r \leq \lim_{(r,|x|) \to \infty } (\tilde{N}_-(|x|,r) + \tilde{N}_-(|x|,r)), $$
and due to \eqref{N+}, we can conclude that 
\begin{align*}
\lim_{(r,|x|) \to \infty } (\tilde{N}_-(|x|,r) + \tilde{N}_-(|x|,r)) &= \lim_{|x| \to \infty } \lim_{r \to \infty}(\tilde{N}_-(|x|,r) + \tilde{N}_-(|x|,r)) \\
& = \lim_{r \to \infty } \lim_{|x| \to \infty}(\tilde{N}_-(|x|,r) + \tilde{N}_-(|x|,r)) \\
& = 0
\end{align*} 
Hence the requirements of Theorem \ref{tucsnak} are satisfied, and we can conclude that any open set $\mathcal{O} \subset (0,2\pi) \times \{ t \in \mathbb{R}: t >0\}$ is a domain associated with $\Lambda$. 
Therefore, if $u(x,t)|_{\mathcal{O}} = 0$, there is some $d_- > 0$  such that 
$$ 0 = \int_{\mathcal{O}} |u(x,t)|^2 \mathrm{d}x\mathrm{d}t \geq d_-\sum_{k \in \mathbb{Z}} |g_k|^2 = d_-\|u(\cdot,t)\|_{L^2}. $$

\end{proof}

Let us end by concluding that the examples mentioned (Schrödinger, linear KdV and gravity-capillary waves) enjoy the unique continuation property, as their dispersion relations all satisfy the asymptotic constraint in Proposition \ref{lin disp}. However, for the gravity-capillary waves, the inclusion of the capillary force, manifested by the term $Sk^3$, is needed for the asymptotic constraint to hold, in contrast to the proof in for the full water waves system. 

Moreover, the characterization of unique continuation in terms of the dispersion relation is interesting in that relates the asymptotically infinite propagation speed to the non-local behaviour of the solution; the implication that zero-sets of the solution cannot be open sets is similar to the behaviour of for example the heat equation, where the solution is everywhere positive and the speed of propagation is infinite.

\printbibliography

\end{document}